\documentclass[10pt,a4paper]{article}
\usepackage[latin2]{inputenc}
\usepackage[T1]{fontenc}
\usepackage{graphicx}
\usepackage{amssymb}
\usepackage{amsmath}
\newtheorem{theorem}{Theorem}[section]
\newtheorem{proposition}[theorem]{Proposition}
\newtheorem{lemma}[theorem]{Lemma}
\newtheorem{corollary}[theorem]{Corollary}

\newtheorem{definition}[theorem]{Definition}
\newtheorem{example}[theorem]{Example}
\newtheorem{remark}[theorem]{Remark}
\newtheorem{property}[theorem]{Property}
\newtheorem{conjecture}[theorem]{Conjecture}
\newtheorem{problem}[theorem]{Problem}
\newcommand{\rank}{\operatorname{rank}}
\newcommand{\ord}{\operatorname{ord}}
\newcommand{\inc}{\operatorname{in}}
\newcommand{\supp}{\operatorname{supp}}
\newcommand{\conv}{\operatorname{conv}}
\newcommand{\inter}{\operatorname{int}}
\newenvironment{proof}{\noindent\textsc{Proof.}} {\nolinebreak[4]\hfill$\blacksquare$\\\par}
\begin{document}
\setcounter{footnote}{0}
\renewcommand{\thefootnote}{}
\footnotetext{\textit{Date}: 27.09.2011
 }
\title{The \L ojasiewicz exponent of  nondegenerate\\ surface singularities}
\author{Grzegorz Oleksik \\ Faculty of Mathematics and Computer Science, University of \L \'od\'z\\
Banacha 22, 90-238 \L\'od\'z, Poland\\ oleksig@math.uni.lodz.pl}

\maketitle

\begin{abstract}
In the article we give some estimations of the \L ojasiewicz exponent of nondegenerate surface singularities in terms of their Newton diagrams. We also give an exact formula for the \L ojasiewicz exponent of such singularities in some special cases. The results are stronger than Fukui inequality \cite{F}. It is also a multidimensional generalization of the Lenarcik theorem \cite{L}. 
\end{abstract}

\section{Introduction}
Let $ f:\left( \mathbb {C}^n,0\right)\longrightarrow \left( \mathbb {C},0\right) $ be a holomorphic function
in an open neighborhood of $0\in \mathbb{C}^n$ and  $\sum_{\nu\in\mathbb{N}^n}{a_{\nu}z^{\nu}}$ be the Taylor expansion of $f$ at $0.$
We define ${\Gamma}_{+}(f):=\mbox{conv}\{\nu +{\mathbb{R}^n_+}:a_{\nu}\neq 0\}\subset\mathbb{R}^n$ and call it
\textit{the Newton diagram} of $f$. Let $u\in\mathbb{R}_+^n\setminus\{0\}$.
 Put $l(u,\Gamma_+(f)):=\inf\{\langle u,v\rangle:v\in\Gamma_+(f)\}$ and $\Delta(u,\Gamma_+(f)):=\{v\in\Gamma_+(f):
 \langle u,v\rangle=l(u,\Gamma_+(f))\}.$ We say that $S\subset\mathbb{R}^n$ is a \textit{face} of $\Gamma_+(f)$ if
 $S=\Delta(u,\Gamma_+(f))$ for some $u\in\mathbb{R}_+^n\setminus\{0\}$. The vector $u$ is called 
\textit{the primitive vector} of $S.$ It is easy to see that $S$ is a closed and convex set and $S\subset
 \mbox{Fr}(\Gamma_+(f)),$ where $\mbox{Fr}(A)$ denotes the boundary of $A.$ One can prove that a face $S\subset\Gamma_+(f)$ is compact if and only if all coordinates of its primitive vector $u$ are positive. We call the family of all compact faces of $\Gamma _{+}(f)$ the \textit{Newton boundary} of $f$ and denote by ${\Gamma}(f)$. We denote by $\Gamma^k(f)$ the set of all compact $k$-dimensional faces of $\Gamma(f),\, k=0,\ldots ,n-1.$ For every compact face $S\in{\Gamma}(f)$
we define quasihomogeneous polynomial $f_S:=\sum_{\nu\in S}{a_{\nu}z^{\nu}}$. We say that $f$ is
    \textit{nondegenerate on a face} $S\in\Gamma(f)$ if the system of equations  
     $(f_S)'_{z_1}=\ldots=(f_S)'_{z_n}=0$ has no solution in
      $({\mathbb{C}^*})^n,\,$ where $\mathbb{C}^*=\mathbb{C}\setminus\{0\}$. We say that $f$ is \textit{nondegenerate in the Kouchnirenko sense} (shortly \it nondegenerate \rm) if it is nondegenerate on each face $\Gamma(f).$ We say that $f$ is a \textit{singularity} if $f$ is a nonzero holomorphic function in some open neighborhood of the origin such that $f(0)=0,\,\nabla f(0)=0,$ where $\nabla f=(f'_{z_1},\ldots ,f'_{z_n}).$ We say that $f$ is \textit{an isolated singularity}  if $f$ is a singularity, which has an isolated critical point at the origin i.e. $\nabla f(z)\neq 0$ for
  $z\neq 0.$
\vspace{0.1cm}
\par\indent Let $i\in\{1,\ldots ,n\},\,n\geq 2.$

\begin{definition}

We say that $S\in\Gamma^{n-1}(f)\subset\mathbb{R}^n$ is \textit{an exceptional face with respect to the axis} $OX_i$ if one of its vertices is at distance $1$ to the axis $OX_i$ and another vertices constitute $(n-2)$-dimensional face which lies in one of the coordinate hyperplane including the axis $OX_i.$

\end{definition}
\begin{figure}[!ht]
\begin{center}
\includegraphics{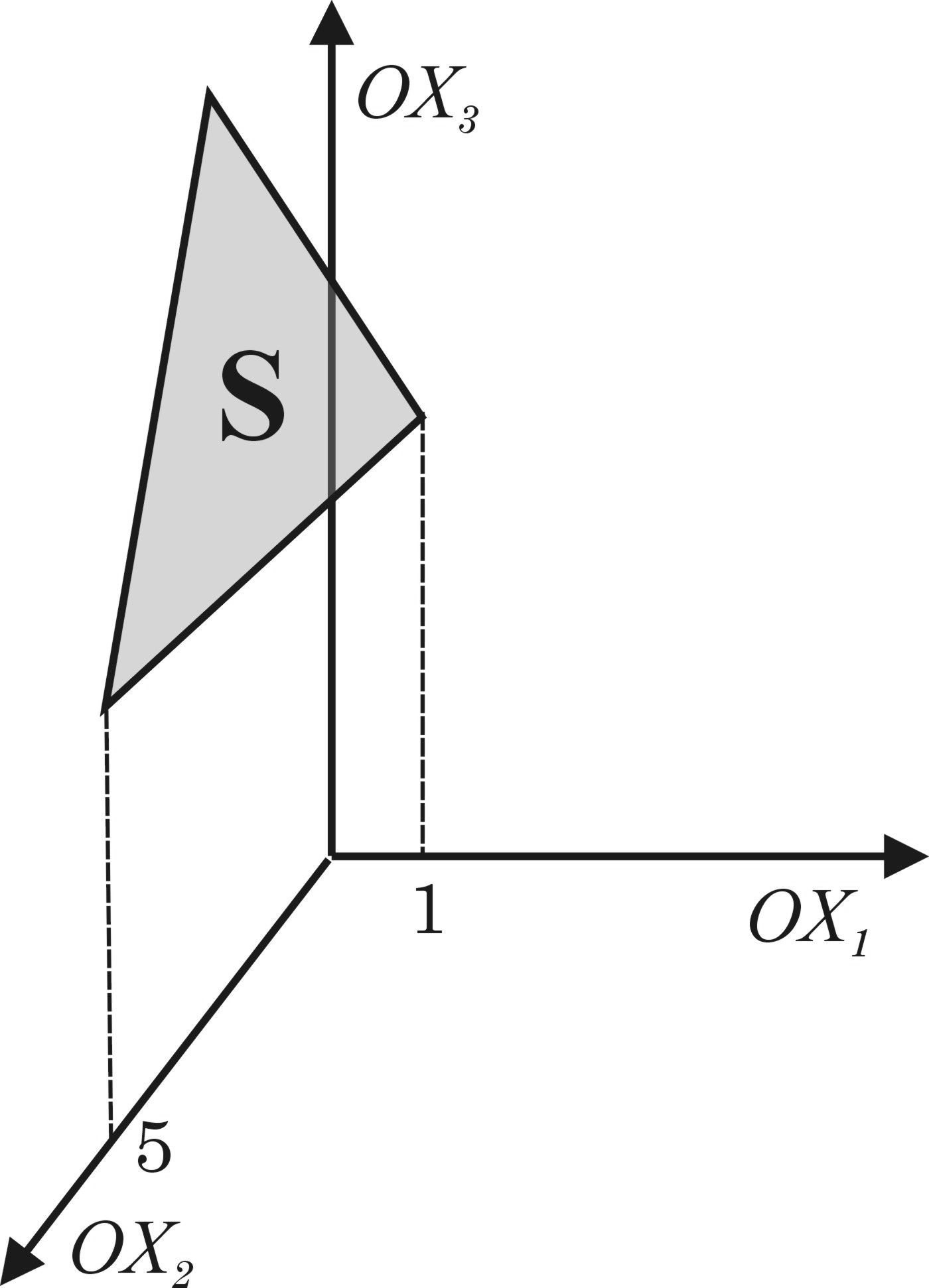}
\caption{An exceptional face $S$ with respect to the axis $OX_3.$}\label{r4.1}
\end{center}
\end{figure}

We say that $S\in\Gamma^{n-1}(f)$ is \textit{an exceptional face of} $f$ if there exists $i\in\{1,\ldots,n\}$ such that $S$ is an exceptional face with respect to the axis $OX_i$. Denote by $E_f$ the set of exceptional faces of $f.$ We call the face $S\in\Gamma^{n-1}(f)$ \textit{unexceptional of} $f$, \rm if $S\not \in E_f.$

\begin{definition} We say that the Newton diagram of $f$ is \textit{convenient} if it has nonempty intersection with every coordinate axis.
\end{definition}


\begin{definition} We say that the Newton diagram of $f$ is \textit{nearly convenient} if its distance to every coordinate axis doesn't exceed 1.
\end{definition}

\par For every $(n-1)$-dimensional compact face $S\in\Gamma(f)$ we shall denote by $x_1(S),\ldots ,x_n(S)$ coordinates of intersection of the hyperplane determined by face $S$ with the coordinate axes. We define $m(S):=\max\{x_1(S),\ldots ,x_n(S)\}.$ It is easy to see that $$x_i(S)=\frac{l(u,\Gamma_+(f))}{u_i},\,i=1,\ldots ,n,$$ where $u$ is a primitive vector of $S.$ It is easy to check that the Newton diagram ${\Gamma}_{+}(f)$ of an isolated singularity $f$ is nearly convenient. So, "nearly convenience"  of the Newton diagram is a neccesary condition for $f$ to be an isolated singularity. For a singularity $f$ such that $\Gamma^{n-1}(f)\neq\emptyset,$ we define  

\begin{equation}m_0(f):=\max _{S\in\Gamma^{n-1}(f)}{m(S)}.\end{equation} 

\noindent It is easy to see that in the case $\Gamma_+(f)$ is convenient $m_0(f)$ is equal to the maximum of coordinates of the points of the intersection of the Newton diagram and the union of all axes.\vspace{0.1cm}


\begin{remark}
A definition of $m_0(f)$ for all singularities ( even for $\Gamma^{n-1}(f)=\emptyset$), can be found in \cite{F}. In the case $\Gamma^{n-1}(f)\neq\emptyset$ both definitions are equivalent.
 \end{remark}

\par Let $f=(f_1,\ldots ,f_n):(\mathbb{C}^n,0)\longrightarrow \left( \mathbb {C}^n,0\right)$ be a holomorphic mapping
having an isolated zero at the origin. We define the number

\begin{equation}l_0(f):=\inf\{\alpha\in\mathbb{R}_+\colon \exists _{C > 0}\exists _{r>0}\forall _{\|z\|<r}\|f(z)\|\geq C\|z\|^{\alpha}
\} \end{equation}

\noindent and call it \textit{the \L ojasiewicz exponent} of the mapping $f.$ There are  formulas and estimations of the number $l_0(f)$ under some nondegeneracy conditions of $f$ (see \cite{A}, \cite{B}, \cite{BE2}, \cite{L}, \cite{O}, \cite{Ph}).
\par Let  $f:(\mathbb{C}^n,0)\longrightarrow \left( \mathbb {C},0\right)$ be an isolated singularity.
We define a number
${\pounds}_0(f):=l_0(\nabla f)$ and call it \textit{the \L ojasiewicz exponent of singularity} $f.$
 Now we give some important known properities of the \L ojasiewicz exponent (see \cite{L-JT}): 
\begin{itemize}
\item[\rm(a)]  ${\pounds}_0(f)$ is a rational number.
\item[\rm(b)]\label{b} ${\pounds}_0(f)=\mbox{sup}\{\frac{\ord \nabla f(z(t))}{\ord z(t)}\colon 0\neq z(t)\in\mathbb{C}
\{t\}^n,\, z(0)=0\}.$
\item[\rm(c)] The infimum in the definition of the \L ojasiewicz exponent  is attained for $\alpha ={\pounds}_0(f)$.
\item[\rm(d)] $s(f)=[{\pounds}_0(f)]+1,$ where $s(f)$ is \textit{ the degree of $C^0$-sufficiency of $f$ } (see \cite{ChL}).
\end{itemize}

\par Lenarcik gave in \cite{L} the formula for the \L ojasiewicz exponent for singularities of two variables, nondegenerate in Kouchnirenko sense, in terms of its Newton diagram (another formulas in two-dimensional case see \cite{CK1}, \cite{CK2}).
  

\begin{theorem}[{\rm[L]}]Let $f:(\mathbb{C}^2,0)\longrightarrow \left( \mathbb {C},0\right)$ be an isolated nondegenerate singularity and $\Gamma^1(f)\setminus E_f\neq\emptyset$. Then \begin{equation}{\pounds}_0(f)=\max_{S\in\Gamma^1(f)\setminus E_f}m(S)-1.
\end{equation}
\end{theorem}


\begin{remark} In two-dimensional case one can prove that for isolated singularities such that $\Gamma^1(f)\setminus E_f=\emptyset ,$
 i.e. $\Gamma^1(f)$ consist of only exceptional segments, we have ${\pounds}_0(f)=1$. 
\end{remark}

\par In multidimensional case we have only an upper bounds for ${\pounds}_0(f),$ which was given by T. Fukui in 1991 (without removing any faces).


\begin{theorem}[{\rm[F]}]\label{fukui} Let $f:(\mathbb{C}^n,0)\longrightarrow \left( \mathbb {C},0\right)$ be an isolated nondegenerate singularity. Then \begin{equation}{\pounds}_0(f)\leq m_0(f)-1.\end{equation} 
\end{theorem}

\par In the paper we improve the Fukui inequality and simultaneously generalize the Lenarcik result (in a weak form) to $3$-dimensional case (by removing exceptional faces). 

We denote by $\overline{AB}$ the segment joining two different points $A,B\in\mathbb{R}^n.$ We consider following segments in $\mathbb{R}^3$: $$I_1^k=\overline{(0,1,1)(k,0,0)},\,I_2^k=\overline{(1,0,1)(0,k,0)},\,I_3^k=\overline{(1,1,0)(0,0,k)},\,\,k\in\{2,3 \ldots \}.$$ Put $\mathcal {J}:=\{I_j^k\colon j=1,2,3,\,k=2,3,\ldots\}.$ Every segment $I$ of this family intersects exactly one coordinate axis in exactly one point. We denote by $m(I)$ nonzero coordinate of this point (equal to $k$). We give now the main result, which is the improvement of the above Theorem \ref{fukui}. 


\begin{theorem}
\label{n2} Let $f:\left( \mathbb {C}^3,0\right)\longrightarrow \left( \mathbb {C},0\right)$ be an isolated and nondegenerate singularity.  
\vspace{0.15cm}

\par\noindent $1^0$ If $\Gamma^{2}(f)=\emptyset$ or $\Gamma^{2}(f)=E_f,$ then there exists excatly one segment $I\in\mathcal{J}\cap\Gamma^{1}(f)$ and $${\pounds}_0(f)=m(I)-1.$$
\par \noindent $2^0$ If $\Gamma^2(f)\setminus E_f\neq\emptyset,$ then \begin{equation}\label{n2'} {\pounds}_0(f)\leq\max _{S\in\Gamma^{2}(f)
 \setminus E_f} m(S)-1. \end{equation} 
\end{theorem} 

In the paper \cite{KOP} there was given formula for the \L ojasiewicz exponent of quasihomogeneous surface singularities in terms of their weights. 

\begin{theorem} \label{Q1} Let $f:\left( \mathbb {C}^3,0\right)\longrightarrow \left( \mathbb {C},0\right)$ be an isolated weighted homogeneous singularity with weights $w_1,w_2,w_3,$ then $${\pounds}_0(f)=\max_{i=1}^3 w_i-1$$
\end{theorem}
(in real case see \cite{HP}). Since in $3$-dimensional case the weights are topological invariants of quasihomogeneous singularities (see \cite{Y}), then from the above formula we get that the \L ojasiewicz exponent is a topological invariant of such singularities.
We can reformulate this result in terms of the Newton diagram as follows.

\begin{theorem}
\label{n3} Let $f:\left( \mathbb {C}^3,0\right)\longrightarrow \left( \mathbb {C},0\right)$ be an
isolated weighted homogeneous singularity.  
\vspace{0.15cm}

\par\noindent $1^0$ If $\Gamma^{2}(f)=\emptyset$ or $\Gamma^{2}(f)$ consist of one exceptional face, then there exists exactly one segment $I\in\mathcal{J}\cap\Gamma^{1}(f)$ and $${\pounds}_0(f)=m(I)-1.$$
\par \noindent $2^0$ If $\Gamma^2(f)$ consists of one unexceptional face $S,$ then \begin{equation*} {\pounds}_0(f)=m(S)-1. \end{equation*} 
\end{theorem} 

So we see the main result of this paper is a generalization of the above theorem (in a weaker form) to non-degenerate case. In 2010 the paper by
Tan, Yau, Zuo (\cite{TYZ}) appeared, in which Theorem \ref{Q1} was given in analogous form for $n$-variables, $n>3,$ but their proof is false (the proof of their Proposition 3.4 is false).  Some results for quasihomogeneous singularities in $n$-dimensional case were also given by Bivia-Ausina and Encinas (\cite{BE1}, \cite{BE2}).  


\section{Auxiliary lemmas and properties.}
 Put $A-1_i:=A-(0,\ldots,\underset{\hat i}1,\ldots,0),\,i=1,\ldots,n$ for every nonempty $A\subset \mathbb{R}^n.$ We give now two simple and useful properties. The proofs are easy, so we omit them.  

\begin{property}
\label{n8} Let $f\in\mathcal{O}^n,\,f(0)=0$ and $\phi=(\phi_i)_{i=1}^n\in\mathbb{C}\{t\}^n$ be a parametrization such that $\phi_i\neq 0,\,i=1,\ldots,n,$ and $L$ be the supporting hyperplane to $\Gamma_+(f)$ such that $w=(\ord\phi_i)_{i=1}^n\bot L.$ 
If $\inc_w f\circ\inc\phi\neq 0,$ then \\
a) $\inc(f\circ \phi)=\inc_w f\circ\inc\phi,\,\ord(f\circ\phi)=\ord_w f,$\\
b) $m(L)=\frac{\ord_w f}{\min_{i=1}^n w_i}=\frac{\ord(f\circ \phi)}{\ord\phi}.$
\end{property}

\begin{property}
\label{n7} Let $f\in\mathcal{O}^n,\,f(0)=0,\,w\in\mathbb{N}^n,\,i\in\{1,\ldots,n\}.$ Suppose that there exists a monomial in $\inc_w(f)$ in which the variable $z_i$ appears, then $$(\inc_w f)'_{z_i}=\inc_w f'_{z_i}.$$ Moreover, if $L$ is the supporting hyperplane to $\Gamma_+(f)$ such that $w\bot L,$ then $L-1_i$ is a supporting hyperplane to $\Gamma_+(f'_{z_i}).$ 
\end{property}

The following lemma will be useful in the proof of Lemma \ref{n6}. 

\begin{lemma} 
\label {n10} Let $f:\left( \mathbb {C}^n,0\right)\longrightarrow \left( \mathbb {C},0\right),\,n\geq 2, $ be a
singularity and $\phi=(\phi _i)_{i=1}^n\in \mathbb{C}\{t\}^n$ be a parameterization such that $\phi _i\neq  
0,\,i=1,\ldots ,n,$ and $w:=(\ord\phi_i)_{i=1}^n.$ Let $$I:=\{i\in \{1,\ldots,n\}:f'_{z_i}\circ\phi= 0 \}\neq\emptyset.$$ Then for the face $S:=\Delta(w,\Gamma_+(f))\in \Gamma(f)$ we get that $(f_S)'_{z_i}\circ \inc\phi=0$ for $i\in I.$ 
\end{lemma}

\begin{proof} 
Put $J:=\{j\in I: S\subset\{(x_1,\ldots ,x_n)\in\mathbb{R}^n:x_j=0\}\}.$
Then for every $i\in I\setminus J$ we can find a monomial in $\inc_w(f)$ in which the variable $z_i$ appears. Therefore we get by Property \ref{n7}  $(\mbox{in}_w f)_{z_i}'=\mbox{in}_w f'_{z_i}$ for $i\in I\setminus J.$ Hence and by Property \ref{n8}a we get for $i\in I\setminus J$ $$0=\mbox{in}_w f'_{z_i}\circ\mbox{in}\phi=(\mbox{in}_w f)_{z_i}'\circ\mbox{in}\phi=(f_S)'_{z_i}\circ\mbox{in}\phi.$$ 
On the other hand $(f_S)'_{z_i}\circ\mbox{in}\phi=0,$ for $i\in J.$  Summing up we obtain that $(f_S)'_{z_i}\circ\rm{in}\phi=0$ for $i\in I.$ 
\end{proof}
 
The following corollaries are direct consequeance of the above lemma. They show that nondegenarate singularity is "near generic" isolated singularity.

\begin{corollary}
\label{n11} Let $f:\left( \mathbb {C}^n,0\right)\longrightarrow \left( \mathbb {C},0\right),\,n\geq 2, $ be a
 singularity and $\phi=(\phi _i)_{i=1}^n\in \mathbb{C}\{t\}^n$ be a parametrization such that $\phi _i\neq  
0,\,i=1,\ldots ,n.$ If $\nabla f\circ\phi=0,$ then there exists face $S\in\Gamma(f)$ such that $\nabla (f_S)\circ\inc\phi=0,$ so $f$ is degenerate on the face $S.$
\end{corollary}

\begin{corollary}
\label{n12} Let $f:\left( \mathbb {C}^n,0\right)\longrightarrow \left( \mathbb {C},0\right),\,n\geq 2, $ be a nondegenarate singularity. If $\phi=(\phi _i)_{i=1}^n\in \mathbb{C}\{t\}^n$ is a parametrization such that $\phi _i\neq 0,\,i=1,\ldots ,n,$ then $\nabla f\circ\phi\neq 0.$ 
 \end{corollary}

\begin{example} Assumptions that $\phi _i\neq 0,\,i=1,\ldots ,n,$ are necessary in the above corollaries. Indeed, let $f(z_1,z_2,z_3)=z_1(z_2+z_3)$ and $\phi(t)=(0,t,-t).$ It is easy to check that $f$ is nondegenerate singularity and $\nabla f\circ \phi=0.$ 
\end{example}

We give now a simple property, which is needed in the proof of the next property.

\begin{property}\label{pr5} Let $f\in\mathcal{O}^n,\,f(0)=0.$ Then $l(u,\Gamma_+(f))\geq \min_{i=1}^n u_i.$ \end{property}

\begin{proof} By definition of $\Gamma_+(f)$ we get that for every $x\in\Gamma_+(f)$ there are exist $a_j\in\supp(f),b_j\in\mathbb{R}_+^n,j=1,\ldots,k,$ and nonnegative real numbers $c_j,j=1,\ldots,k,\,\sum_{j=1}^kc_j=1,$ such that $x=\sum_{j=1}^kc_j(a_j+b_j).$ We have further  
\begin{eqnarray}\label{pri}\langle u,x\rangle &=& \sum_{i=1}^n u_ix_i=\sum_{i=1}^nu_i\left(\sum_{j=1}^kc_j[(a_j)_i+(b_j)_i)]\right)\geq(\min_{i=1}^nu_i)\cdot \\ \nonumber &\cdot & \sum_{i=1}^n\left(\sum_{j=1}^kc_j[(a_j)_i+(b_j)_i]\right)= (\min_{i=1}^nu_i)\sum_{j=1}^kc_j\left(\sum_{i=1}^n[(a_j)_i+(b_j)_i]\right).
\end{eqnarray}
 Since $f(0)=0$ and $a_j\in\supp(f),j=1,\ldots k,$ so $(a_j)_i\geq 1$ for some $i\in\{1,\ldots n\}$ dependent from $j,$ thus $\sum_{i=1}^n[(a_j)_i+(b_j)_i]\geq 1,\,j=1,\ldots,k.$ This and inequality (\ref{pri}) shows that $\langle u,x\rangle\, \geq \min_{i=1}u_i$ for $x\in\Gamma_+(f),$ so $l(u,\Gamma_+(f))\geq \min_{i=1}^n u_i.$ It finishes the proof. \end{proof}  

We give now useful property, which will be often used in the next part of the paper. 

\begin{property}
\label{n4'}Let $f\in\mathcal{O}^n,\,f(0)=0,$ and $L$ be a supporting hyperplane to a compact face of $\Gamma_+(f).$ Then $m(L)\geq 1.$ Morover if $f'_{z_i}(0)=0$ and $L-1_i$ supports a compact face of $\Gamma_+(f'_{z_i}),$ then $m(L-1_i)\geq 1,\,m(L)\geq 2$ and $$m(L-1_i)\leq m(L)-1,\,i\in\{1,\ldots,n\}.$$ 
\end{property}

 \begin{proof} 
Let $u=(u_j)_{j=1}^n$ be a supporting vector to the hyperplane $L$ and also $L-1_i$ and let $b:=l(u,\Gamma_+(f)),\,b_i:=l(u,\Gamma_+(f'_{z_i})).$ Then $L,L-1_i$ have respectively equations $\sum_{j=1}^n u_j x_j=b,\,\, \sum_{j=1}^n u_j x_j=b_i.$ Because $L,L-1_i$ support compact faces, so $u_j>0,\,j=1,\ldots,n,$ and then by Property \ref{pr5} we have $b,b_i>0.$ On the other hand hyperplane $L-1_i$ has equation $\sum_{j=1}^nu_j(x_j+\delta_{ij})=b,$ where $\delta_{ij}$ is the Kronecker symbol. Hence $b-u_i=b_i>0.$ Therefore we get   
\begin{equation} \label{n4''} m(L)=\max_{j=1}^n\frac{b}{u_j},\,m(L-1_i)=\max_{j=1}^n\frac{b_i}{u_j}=\max_{j=1}^n\frac{b-u_i}{u_j}. \end{equation} 
Hence by Property \ref{pr5} we have that $m(L),m(L-1_i)\geq 1.$ By (\ref{n4''}) we get \begin{eqnarray*}m(L-1_i)&=&\max_{j=1}^n\frac{b-u_i}{u_j}=\frac{b-u_i}{\min_{j=1}^n u_j}\leq\frac{b-\min_{j=1}^n u_j}{\min_{j=1}^n u_j}=\frac{b}{\min_{j=1}^n u_j}-1\leq \\
& \leq & \max_{j=1}^n\frac{b}{u_j}-1=m(L)-1. \end{eqnarray*} Because $m(L-1_i)\geq 1,$ so from the last inequality we have that $m(L)\geq 2.$ It finishes the proof. 
\end{proof}

We give now the lemma, which is important in the second part of proof of the Theorem \ref{n2}. It shows a method to find an  upper bound of the \L ojasiewicz exponent of nondegenerate singularity in terms of its Newton diagram. 
 
\begin{lemma}
\label{n6} Let $f:\left( \mathbb {C}^n,0\right)\longrightarrow \left( \mathbb {C},0\right)$ be a nondegenerate
 singularity. Let $\phi=(\phi_i)_{i=1}^n\in\mathbf{C}\{t\}^n$ be a parametrization such that $\phi_i\neq 0,\,i=1,2,\ldots,n,$ and $L$ be the supporting hyperplane to $\Gamma_+(f)$ such that $w:=(\ord\phi_i)_{i=1}^n\bot L.$ Then $\nabla f\circ\phi\neq 0$ and $$\frac{\ord(\nabla f\circ\phi)}{\ord\phi}\leq m(L)-1.$$
\end{lemma}
 
\begin{proof}
By Corollary \ref{n12} we get that $\nabla f\circ\phi\neq 0.$ From our assumption $w$ has positive coordinates, so $L$ is a supporting hyperplane to a compact face $S\in\Gamma(f).$ Set $$J:=\{ i\in \{1,\ldots,n\}  \colon (f_S)'_{z_i}\circ \inc \phi \neq 0 \},\,K:= \{ i\in \{1,\ldots,n\} \colon f'_{z_i}\circ\phi\neq 0\}.$$ Because $f$ is a nondegenerate singularity, so $J\neq \emptyset.$ By Lemma \ref{n10} we get that $J \subset K.$ Observe that, for $i\in J$ we can find a monomial in $f_S$ in which the variable $z_i$ appears and then by Property \ref{n7} $L-1_i$ supports compact face of $\Gamma_+(f'_{z_i})$ and $\inc_w f'_{z_i}=(\inc_w f)'_{z_i}.$ Hence and because $f_S=\inc_w f,$ we get for $i\in J$  $$\inc_w f'_{z_i}\circ\inc\phi=(\inc_w f)'_{z_i}\circ\inc\phi=(f_S)'_{z_i}\circ\inc\phi\neq 0.$$ Hence and because $J \subset K,$ so by Property \ref{n4'} and Property \ref{n8} we get, that
\begin{eqnarray*}
  \frac{\ord(\nabla f\circ\phi)}{\ord\phi} & = & \min_{i\in K}\frac{\ord( f'_{z_i}\circ\phi)}{\ord\phi}\leq\min_{i\in J}\frac{\ord( f'_{z_i}\circ\phi)}{\ord\phi}=\min_ {i \in J}\frac{\ord_w f'_{z_i}}{\min_{i=1}^n w_i}= \\
  &=&\min_{i\in J} m(L-1_i)\leq m(L)-1.
\end{eqnarray*}
It concludes the proof. \end{proof}

The following property in this section says that the Newton boundary of the restriction $f|_{\{z_{k+1}=\ldots =z_n=0\}}$ is the restriction of the Newton boundary of $f$ to the set $\{x_{k+1}=\ldots =x_n=0\}\subset\mathbb{R}^n.$ 

\begin{property}\label{w4} Let $f\in\mathcal{O}^n,\,n\geq 2.$ Assume that $g(z_1,\ldots ,z_k):=f(z_1,\ldots ,z_k,$ \\ $0,\ldots,0)\in\mathcal{O}^k,\,k<n,$ is a nonzero germ. Then 
\begin{equation}\Gamma(g)=\left\{S\in \Gamma(f):S\subset \{x_{k+1}=\ldots =x_n=0\}\right\}.\end{equation}
\end{property}

 \begin{proof}
 $"\subset".$ Let $S\in \Gamma(g),$ then $ S=\Delta(u,\Gamma_+(g))$ for some 
 $u\in(\mathbb{R}_+\setminus \{0\})^k.$ Of course $S\subset\Gamma_+(f)\cap \{x_{k+1}=\ldots =x_n=0\}.$ Set
 $$u'=(u_1,\ldots,u_k,l(u,\Gamma_+(g))+1,\ldots,l(u,\Gamma_+(g))+1)\in \mathbb{R}^n.$$ We show that $S=\Delta(u',\Gamma_+(f))
 .$ By definition of $u'$ we have that $l(u',\Gamma_+(f))$ can be realized only for $v\in\Gamma_+(f)\cap \{x_{k+1}= \ldots =x_n=0\}.$ On the other hand it is easy to check that $$\Gamma_+(f)\cap \{x_{k+1}=\ldots =x_n=0\}=\Gamma_+(g).$$ So we get $l(u',\Gamma_+(f))=l(u,\Gamma_+(g))$ and $\Delta(u',\Gamma_+(f))=\Delta(u,\Gamma_+(g)).$ Summing up we obtain $S=\Delta(u',\Gamma_+(f)),$ so $S\in\Gamma(f).$ 

\vspace{0.15cm} 

\par\noindent $"\supset ".$ Let $S\in\Gamma(f)$ and $S\subset\{x_{k+1}=\ldots =x_n=0\}.$ Then $S=\Delta(u,\Gamma_+(f))$
 for some $u\in(\mathbb{R}_+\setminus\{0\})^n$ and as we observed above $\Gamma_+(f)\cap \{x_{k+1}=\ldots =
 x_n=0\}=\Gamma_+(g).$ So $l(u,\Gamma_+(f))=l(u',\Gamma_+(g)),$ where $u'=(u_1,\ldots,u_k).$ It follows that $\Delta(u',\Gamma_+(g))=\Delta(u,\Gamma_+(f)).$ Hence $S=\Delta(u',\Gamma_+(g)),$ so $S\in\Gamma(g).$ That concludes the proof. 
\end{proof}   

We give now an interesting property needed in the next part of the paper. 

\begin{property}
\label{n17} Let $f\in\mathcal{O}^n,\,n\geq 3,$ be an isolated singularity. Then $f$ is an irreducible germ in $\mathcal{O}^n.$ 
\end{property}

\begin{proof}
Suppose to the contrary, that $f=gh,$ where $g$ and $h$ are non-invertible in $\mathcal{O}^n.$ Then we have $f'_{z_i}=g'_{z_i}h+h'_{z_i}g,\,i=1,2,\ldots,n.$ Hence $V(g,h)\subset V(\nabla f).$ Since $g$ and $h$ are non-invertible, then $V(g,h)\neq\emptyset,$ because $0\in V(g,h).$ Hence by Corollary 8 \cite[p. 81]{G} and because $n\geq 3,$ we have that $\dim V(g,h)\geq 1.$ Therefore $\dim V(\nabla f)\geq 1,$ so $\nabla f$ hasn't an isolated zero, a contradiction. 
\end{proof}
\begin{remark} There exist reducible isolated singularities of two variables, e.g. $f(z_1,z_2)=z_1z_2.$ \end{remark}

The following corollary is a direct consequence of Property \ref{n17}.

\begin{corollary}
\label{n18} Let $f\in\mathcal{O}^n,\,n\geq 3,$ be an isolated singularity. Then $$\{x\in\mathbb{R}^n:\,x_i=0\} \cap \Gamma(f) \neq \emptyset,\,i=1,\ldots,n.$$ 
\end{corollary}

\begin{proof} 
Indeed, if $\{x\in\mathbb{R}^n:\,x_i=0\} \cap \Gamma(f)= \emptyset,$ then the singularity $f$ would be represented as $z_ig(z_1,\ldots,z_n),$ where $g$ is a holomorphic function.  Therefore $f$ would be reducible, which by Property \ref{n17} isn't possible.
\end{proof}  
The last lemma says, when the Milnor number is equal to the \L ojasiewicz exponent.
\begin{lemma}{\rm (see \cite{P1})}\label{milnor} Let $f:\left( \mathbb{C}^n,0\right)\longrightarrow \left( \mathbb {C},0\right) $ be an isolated singularity. Then $\pounds_0(f)\leq \mu_0(f).$ If additionally $\rank [f''_{z_iz_j}]_{i,j=1}^n(0)\geq n-1,$ then $\pounds_0(f)=\mu_0(f).$ 
\end{lemma} 


\section{A lemma about the choice of an unexceptional face.}

We give now the lemma, which associates to every coordinate axis in $\mathbb{R}^3$ a suitable unexceptional face of $f$. It turns out to be the main tool in the proof of part $2^0$ of the main result.   

\begin{lemma}[About the choice of an unexceptional face.]\label{wybor} Let $f\in\mathcal{O}^3$ be an isolated singularity such that $\Gamma^2(f)\setminus E_f\neq \emptyset.$ Then for every axis $OX_i,\,i=1,2,3,$ there exists a face $S_i\in\Gamma^2(f)\setminus E_f$ such that at least one of the two conditions is true:
\vspace{0.1cm}
\\$\boldsymbol{i}$) there exists a point $W \in OX_i,$ which is a vertex of the face $S_i,$
\\$\boldsymbol{ii}$) there exist $j,k\in\{1,2,3\}\setminus \{i\},\,j \neq k$ and vertices: $W \in OX_iX_j$ such that its distance to the axis $OX_i$ is equal to $1$ and $Y \in OX_iX_k$ such that segment $\overline{WY}$ is an edge of the face $S_i.$ 
\end{lemma}

Before we pass to the proof we give some properties, lemmas and auxiliary facts. We begin with a simple property of the vertices of the Newton boundary. 

\begin{property}
\label{k1} Let $f\in\mathcal{O}^n$ and $A\in\Gamma^0(f).$ Then $(A+\mathbb{R}^n_+)\cap\Gamma^0(f)=\{A\}.$
 \end{property}
 
\begin{proof}
 Suppose to the contrary there exists an vertex $B\in (A+\mathbb{R}^n_+)\cap\Gamma^0(f),\,B\neq A.$ Hence $B=A+x_0,\,x_0\in \mathbb{R}^n_+\setminus \{0\}.$ Because $B\in\Gamma^0(f),$ so there exists  $u\in \mathbb{R}^n_+\setminus \{0\}$ and a supporting hyperplane $L: \langle u,x\rangle=l(u,\Gamma_+(f))>0,$ to $\Gamma_+(f)$ in the point $B.$ Denote $l:=l(u,\Gamma_+(f)).$ Hence $L\cap \Gamma_+(f)=\{B\}$ and $\langle u,x\rangle >l$ for $x\in \Gamma_+(f)\setminus \{B\}.$ In  particular $\langle u,A\rangle >l.$ On the other hand
$$l=\langle u,B\rangle=\langle u,A+x_0\rangle=\langle u,A\rangle+\langle u,x_0\rangle\geq\langle u,A\rangle,$$ a contradiction. 
\end{proof}

The following property says that segments joining vertices, which lie "properly near" to the coordinate axes are edges. Denote by $x_i(A)$ the $i$-coordinate of the point $A\in\mathbb{R}^n.$ 

\begin{property}
\label{n32} Let $f\in\mathcal{O}^3$ be a singularity and $\{i,j,k\}$ be a permutation of the set $\{1,2,3\}.$ Suppose that there exists a point $W\in\Gamma^0(f)\cap OX_iX_k$ at distance $1$ to the axis $OX_i$ and $\Gamma^0(f)\cap OX_iX_j\neq \emptyset.$ If $Y\in \Gamma^0(f)\cap OX_iX_j$ is the point with the smallest distance to the axis $OX_i,$ then $\overline{WY}\in\Gamma^1(f).$
\end{property} 

\begin{proof} Without loss of generality we may assume that $i=1,\,j=2,\,k=3.$ If $Y\in OX_1,$ then $\overline{WY}\subset OX_1X_3$ and by Property \ref{w4} $\overline{WY}\in \Gamma^1(g)\subset\Gamma^1(f),$ where $g=f|_{\{z_2=0\}}.$ If $Y\not\in OX_1,$ then to get assertion it suffices to find a supporting plane to  $\Gamma_+(f)$ on the segment $\overline{WY}.$ To this end we first observe that planes going through $\overline{WY}$ can intersect the axis $OX_1$ arbitrary far away. Therefore we can choose a vector $u$ and plane $L: \langle u,x\rangle=b>0$ going through $\overline{WY}$ such, that: $1<x_3(L)<2$ and $x_2(Y)<x_2(L)<x_2(Y)+1$ and $P=(0,1,1)$ lie above $L$ (see Fig. \ref{lem}). Hence and since the points of $\supp f$ have integral every coordinate, then $L\cap\Gamma_+(f)=\overline{WY}$ and there is no points of $\supp f$ below plane $L.$ So $\langle u,x\rangle >b$ for every $x\in \Gamma_+(f)\setminus \overline{WY}.$ Summing up $b=l(u,\Gamma_+(f))$ and $L$ is a supporting plane to the edge $\overline{WY}.$ It finishes the proof.
\end{proof}  

\begin{figure}[!ht]
\begin{center}
 \includegraphics{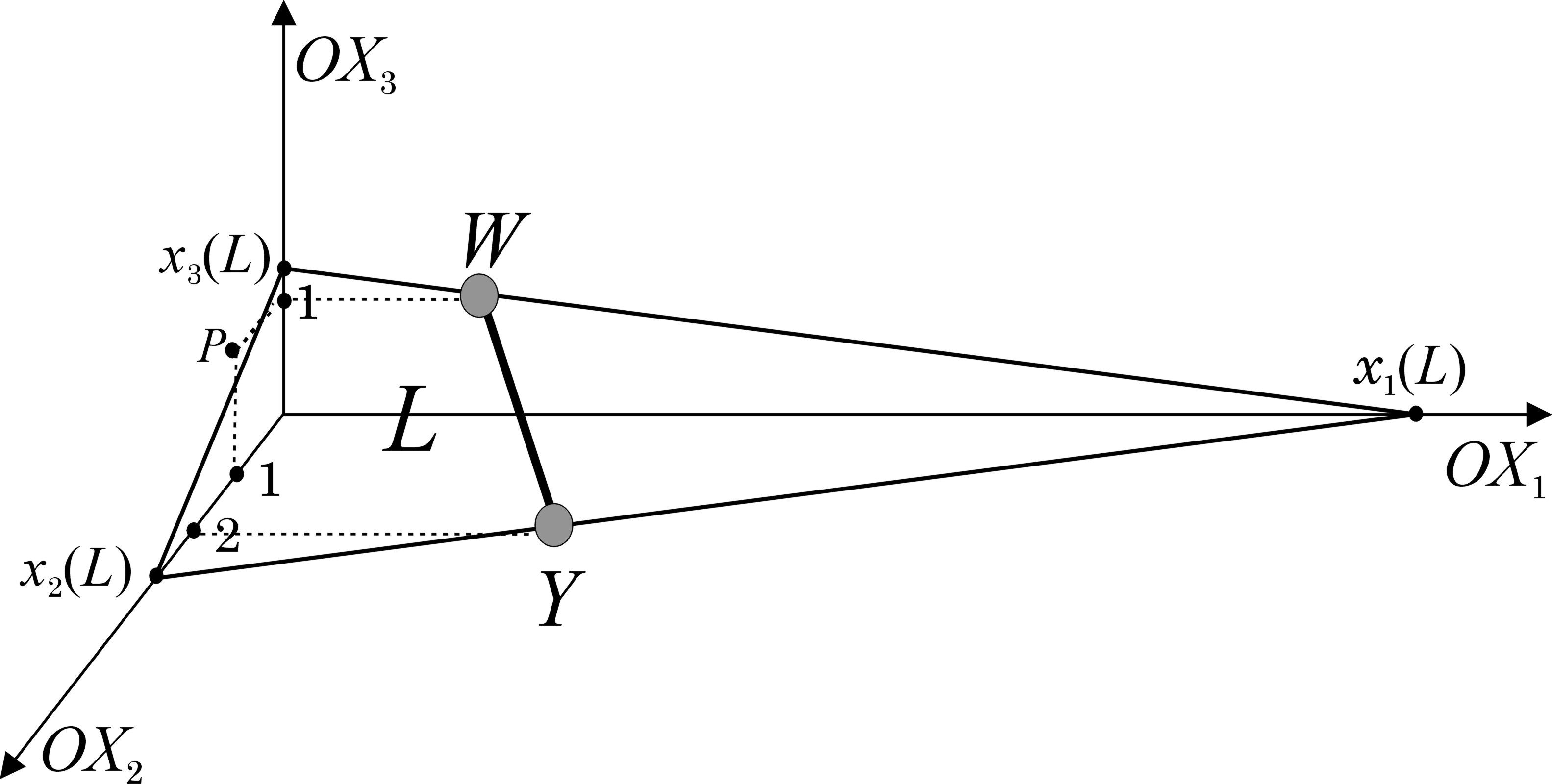} 
  \caption{Plane $L$ supports $\Gamma_+(f)$ on the segment $\overline{WY}.$} \label{lem}
  \end{center}
\end{figure} 

Remind that we have already defined family $\mathcal J=\{I_j^k\colon j=1,2,3,\,k=2,3,\ldots\}.$

\begin{proposition}
\label{k2} Let $f\in\mathcal{O}^3.$ Then the set $\mathcal{J}\cap\Gamma^1(f)$ is empty or consists of one element. Moreover  if $\mathcal{J}\cap\Gamma^1(f)\neq \emptyset,$ then either $\Gamma^2(f)=\emptyset$ or $\Gamma^2(f)=E_f\neq\emptyset.$ 
\end{proposition}

\begin{proof}
 Suppose that $\mathcal{J}\cap\Gamma^1(f)\neq \emptyset.$ Without loss of generality we may assume that $I_3^k\in\mathcal{J}\cap\Gamma^1(f)$ for some $k\in\{2,3,\ldots\}.$ Let $A=(1,1,0),\,B=(0,0,k)$ be vertices of the segment $I_3^k.$ By Lemma \ref{k1} $(A+\mathbb{R}^3_+)\cap\Gamma^0(f)=\{A\}.$ Since the vertices of $\Gamma^0(f)$ have integral coordinates, then $\Gamma^0(f)\setminus\{A\}\subset (OX_1X_3\cup OX_2X_3).$
\par If $\Gamma^0(f)=\{A,B\},$ then $\Gamma^2(f)=\emptyset$ and $\Gamma^1(f)=\{I^k_3\}.$ 
\par Otherwise $\Gamma^2(f)\neq\emptyset$ and joining point $A$ with points of $\Gamma^0(f)\cap OX_1X_3$ and with points of $\Gamma^0(f)\cap OX_2X_3$ we get each face of $\Gamma^2(f)$ and they are all exceptional. Observe that in this case we have $\Gamma^1(f)\cap\mathcal{J}=\{I^k_3\}.$ It finishes the proof. 
\end{proof}   

\begin{lemma}
\label{k3} Let $f\in\mathcal{O}^3$ be an isolated singularity and $\{i,j,k\}$ be a permutation of the set $\{1,2,3\}.$ Moreover let $\mathcal{S}$ be a nonempty family of the exceptional faces with respect to the axis $OX_i.$ Suppose that they all have common vertex $W\in OX_iX_j$ at distance $1$ to the axis $OX_i$ and their another vertices lie in the plane $OX_iX_k,$ and let vertex $Y$ be the one with the smallest distance to the axis $OX_k.$ Then the segment $\overline{WY}$ is either edge of some unexceptional face of $f$ or $\overline{WY}\in \mathcal{J}$ (see Fig. \ref{lem236}).  
\end{lemma}

\begin{figure} [!ht]
\begin{center}
 \includegraphics{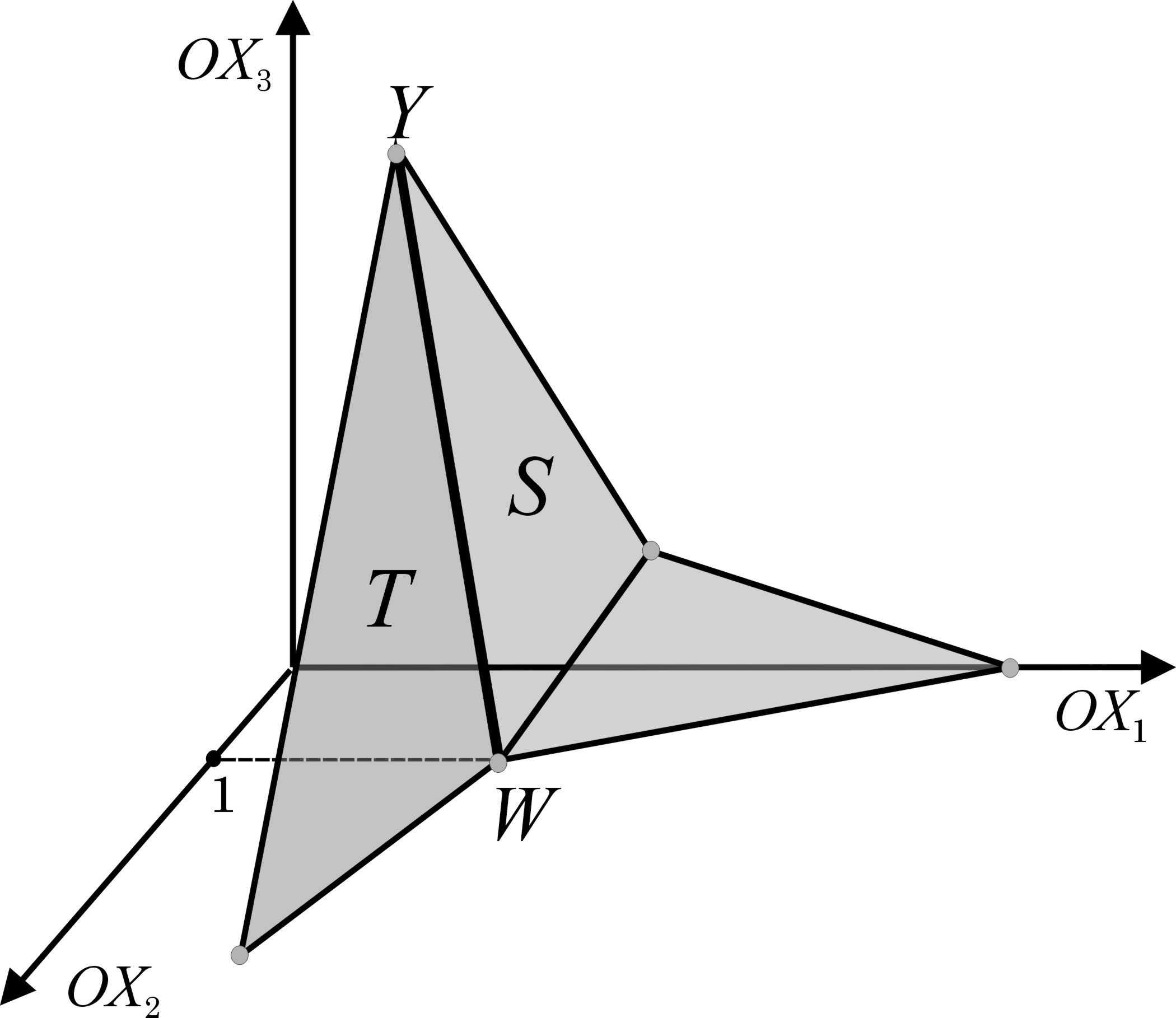} 
  \caption{$\overline{WY}$ is a common edge of an exceptional face $S$ and an unexceptional face $T.$}\label{lem236}
  \end{center}
\end{figure}

\begin{proof} Without loss of generality we may assume that $i=1,\,j=2,\,k=3.$ From the assumption the segment $\overline{WY}$ is edge of an exceptional face $S\in\mathcal{S}.$ 
\par Suppose that $\overline{WY}$ isn't edge of any other face $T\in\Gamma^2(f).$ In particular $\overline{WY}$ isn't edge of any unexceptional face. Then by nearly convenience of $\Gamma_+(f)$ we have $W=(1,1,0)$ and since $\Gamma(f)\cap OX_2X_3\neq\emptyset$ (see Corollary \ref{n18}), then $Y\in OX_3.$ Hence $\overline{WY}\in\mathcal{J}.$ 
\par Suppose now that segment $\overline{WY}$ is also an edge of a face $T\in\Gamma^2(f),\,T\neq S.$ Then by its definition we have that either $T\in\Gamma^2(f)\setminus E_f$ or $T$ is exceptional with respect to an axis different from $OX_1.$ If $T\in\Gamma^2(f)\setminus E_f,$ then we get the thesis. So  suppose that face $T$ is exceptional with respect to an axis different from $OX_1.$ Because $W\not \in OX_1X_3$ and $W\not \in OX_2X_3,$ so $\overline{WY}$ couldn't be edge of an exceptional face with respect to $OX_3.$ Therefore the face $T$ is exceptional with respect to the axis $OX_2.$ Because $\overline{WY}\not\subset OX_1X_2$ and $\overline{WY}\not\subset OX_2X_3,$ so one of the vertices ($W$ or $Y$) is at distance $1$ to the axis $OX_2.$ Because $Y\in OX_1X_3$ and $f$ is a singularity, so $Y$ can't be at distance $1$ to the axis $OX_2.$ Therefore $W$ is at distance $1$ to the axis $OX_2$ and moreover $Y\in OX_2X_3.$ Hence $W=(1,1,0)$ and $Y\in OX_3,$ vis $\overline{WY}\in\mathcal{J}.$ It finishes the proof.          
\end{proof}

Directly by Proposition \ref{k2} and Lemma \ref{k3} we get the following lemma, which turns out to be the key in the proof of the Lemma about the choice of an unexceptional face.

\begin{lemma}
\label{k4} Let $f\in\mathcal{O}^3$ be an isolated singularity such that $\Gamma^2(f)\setminus E_f\neq \emptyset.$ Moreover let $\{i,j,k\}$ be a permutation of the set $\{1,2,3\}$ and $\mathcal{S}$ be a nonempty family of exceptional faces with respect to the axis $OX_i.$ Suppose that they all have common vertex $W\in OX_iX_j$ at distance $1$ to the axis $OX_i$ and their another vertices lie in a plane $OX_iX_k,$ and let vertex $Y$ be the one with the smallest distance to the axis $OX_k.$ Then the segment $\overline{WY}$ is an edge of some unexceptional face of $f$ (see Fig. \ref{lem236}).  
\end{lemma}

Denote by $\#F$ the number of elements in a finite set $F.$ The following proposition says how does look like the Newton boundary of isolated singularities,  which have no $2$-dimensional faces. 

\begin{proposition}
\label{n61}Let $f\in\mathcal{O}^3$ be an isolated singularity. If $\Gamma^2(f)=\emptyset,$ then exists $I\in\mathcal{J},$ such that $\Gamma^1(f)=\{I\}.$ 
\end{proposition}

\begin{proof} If $\Gamma^2(f)=\emptyset,$ then of course $\Gamma^1(f)$ consists of only one segment $I=\overline{AB}$ and $\Gamma^0(f)\subset I.$ Therefore $\Gamma^0(f)=\{A,B\}.$ Let $i\in\{1,2,3\}.$  Denote by $N_i$ the set of vertices, which lie on the axis $OX_i$ or be at distance $1$ to its. By nearly convenience of  $\Gamma_+(f)$ we get that $N_i\neq\emptyset.$ If $N_i$ are pairwise, then $\#\Gamma^0(f)\geq 3,$ a contradiction. Hence there exist $j,k\in\{1,2,3\},\,j\neq k$ such that $N_j\cap N_k\neq \emptyset.$ Without loss of generality we may assume that $j=1$ and $k=2.$ Let $W\in N_1\cap N_2.$ Because $f$ is a singularity, then $W\not\in OX_m,\,m=1,2,3.$ Hence $W=(1,1,0)\in\{A,B\}.$ Without loss of generality we may assume that $W=A.$ Then by Corollary \ref{n18} we get that $B\in OX_3.$ Since $f$ is a singularity, then $k:=x_3(B)\geq 2.$ Summing up $I=I^k_3\in \mathcal{J}.$ It finishes the proof.
\end{proof}

We give now a simple condition to decide, when all $2$-dimensional faces of the Newton boundary are exceptional or there is no any $2$-dimensional faces. It is a "border" case, in which the assumption $1^0$ of the main result is true. 

\begin{theorem}
\label{n19} Let $f\in\mathcal{O}^n$ be an isolated singularity. Then 
$$(\Gamma^2(f)=\emptyset \quad \mbox{or}\quad \Gamma^2(f)=E_f\neq\emptyset ) \Longleftrightarrow \Gamma^1(f)\cap \mathcal{J}\neq \emptyset.$$
\end{theorem}

\begin{proof} 

\vspace{0.1cm}

\noindent "$\Rightarrow$". If $\Gamma^2(f)=\emptyset,$ then by Proposition \ref{n61} we get that $\Gamma^1(f)\cap \mathcal{J}\neq \emptyset.$ Suppose now that $\Gamma^2(f)=E_f\neq \emptyset.$ Let $S\in E_f.$ Without loss of genarality we may assume that $S$ is exceptional with respect to the axis $OX_3.$ Let $\mathcal{S}$ be family of all exceptional faces with respect to the axis $OX_3.$ There exists a vertex $W$ at distance $1$ to axis $OX_3$ which is a common vertex of this family. Without loss of genarality we may supposse that $W\in OX_1X_3$ and another vertices lie in plane $OX_2X_3$ and let edge $Y$ be the one with the smallest distance to the axis $OX_2.$ Then by Lemma \ref{k3} $\overline{WY}\in \mathcal{J}.$    

\vspace{0.1cm}

\noindent "$\Leftarrow$". It is a direct consequence of Proposition \ref{k2}.
\end{proof} 

\vspace{0.2cm}

We can now give the proof of Lemma about the choice of an unexceptional face. For convenience of a reader we repeat its once again.
\vspace{0.2cm}

\par\noindent\bf Lemma \ref{wybor}(\rm \bf About the choice of an unexceptional face.) \it Let $f\in\mathcal{O}^3$ be an isolated singularity such that $\Gamma^2(f)\setminus E_f\neq \emptyset.$ Then for every axis $OX_i,\,i=1,2,3,$ there exists a face $S_i\in\Gamma^2(f)\setminus E_f$ such that at least one of the two conditions is true:
\vspace{0.1cm}
\\$\boldsymbol{i}$) there exists a point $W \in OX_i,$ which is a vertex of the face $S_i,$
\\$\boldsymbol{ii}$) there exist $j,k\in\{1,2,3\}\setminus \{i\},\,j \neq k$ and vertices: $W \in OX_iX_j$ is at distance $1$ to the axis $OX_i$ and $Y \in OX_iX_k$ such that segment $\overline{WY}$ is an edge of the face $S_i.$  \rm
\vspace{0.2cm}

\begin{proof} Let $i\in\{1,2,3\}.$ Without loss of genarality we may assume that $i=3.$ By nearly convenience of $\Gamma_+(f)$ there exists a vertex, which lies on the axis $OX_3$ or at distance $1$ to it. If there exists a vertex, which lies on axis $OX_3,$ then we will denote it by $W_3.$ If there exists a vertex of the Newton boundary at distance $1$ to the axis $OX_3$ and which lies on the plane $OX_iX_3,$ then we will denote it by $W_i,\,i=1,2.$ We have the following cases.
\\$\mathbf{1^0}$ \rm 
There exists a vertex $W_3$ and there aren't vertices $W_1$ and $W_2.$ If $W_3$ is a vertex of some unexceptional face, then the condition i) is fulfilled for this face. Otherwise it is a vertex of some exceptional face $T.$ Since there aren't vertices $W_1$ and $W_2,$ then it is an exceptional face with respect to the axis $OX_1$ or $OX_2.$ Without loss of generality we may assume that it is exceptional with respect to $OX_1.$ Then there exists vertex $B\in OX_1X_2$ at distance $1$ to the axis $OX_1.$ By Lemma \ref{k4} the segment $\overline{BW_3}$ is an edge of some unexceptional face $S_3$, so the condition i) is fulfilled for this face (Fig. \ref{lemwyb}, $1^0$).

\begin{figure}[!ht]
\begin{center}
 \includegraphics{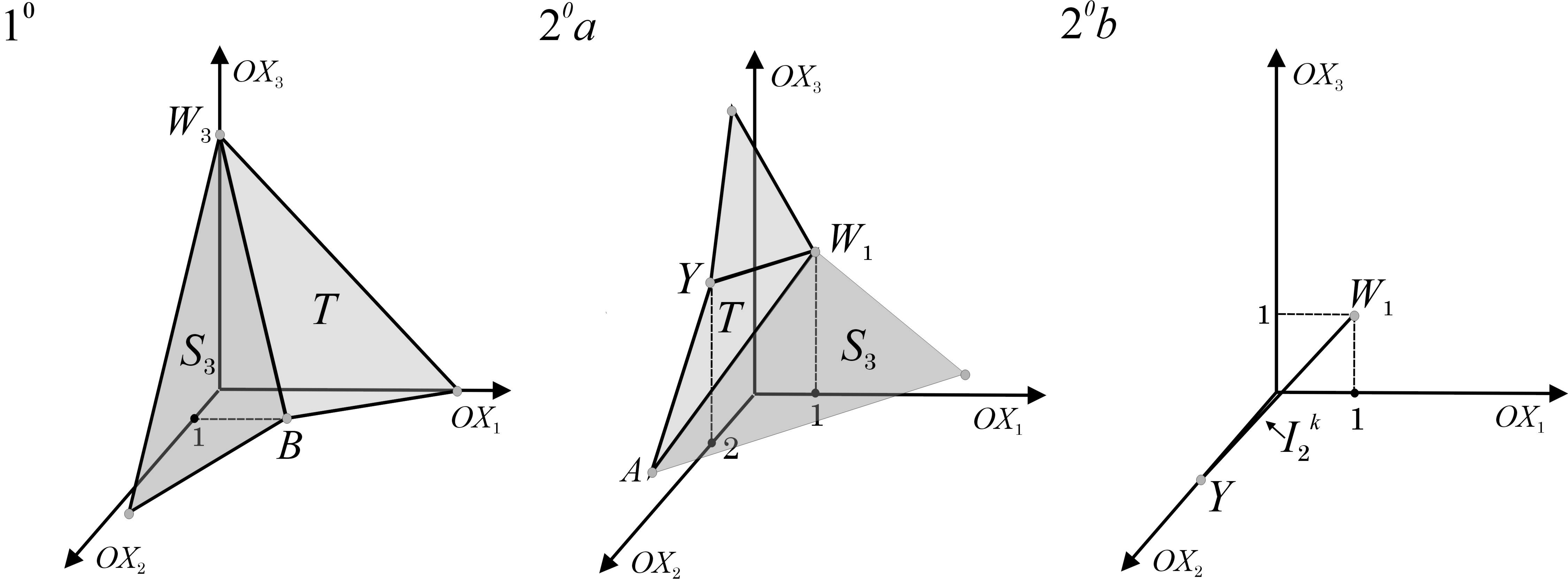} 
  \caption{Unexceptional faces satisfying i) or ii) and segment $I^k_2.$ }\label{lemwyb}
  \end{center}
\end{figure}

\par\noindent $\mathbf{2^0}$ \rm Suppose now that there exists vertex $W_1$ or $W_2.$ Without loss of generality we may assume that there exists vertex $W_1.$ By Corollary \ref{n18} we get that $\Gamma(f)\cap OX_2X_3\neq \emptyset.$ Let vertex $Y\in OX_2X_3\cap\Gamma^0(f)$ be the one with the smallest distance to the axis $OX_3.$ By Property \ref{n32} segment $\overline{W_1Y}\in\Gamma^1(f).$ If it is the edge of some unexceptional face, then the condition ii) is fulfilled for this face. Otherwise it is an edge of some exceptional face. We have the following cases.
\\\bf a) \rm $Y\not\in OX_2$ and $Y\not \in OX_3.$ Then the segment $\overline{W_1Y}$ can't be an edge of any exceptional face with respect to $OX_2$ or $OX_1.$ Therefore it is an edge of some exceptional face $T$ with respect to the axis $OX_3.$ 
\par If $x_2(Y)>1,$ then by Lemma \ref{k4} there exists a vertex $A\in OX_2X_3$ such that the segment $\overline{AW_1}$ is an edge of some unexceptional face $S_3$, so condition ii) is fulfilled for this face (Fig. \ref{lemwyb}, $2^0$a).
\par If $x_2(Y)=1,$ to $Y=W_2.$ Hence by Lemma \ref{k4} there exists $i\in\{1,2\}$ and a vertex $A_i\in OX_iX_3$ such that the segment $\overline{A_iW_{3-i}}$ is an edge of some unexceptional face, so condition ii) is fulfilled for this face. 
\\\bf b) \rm $Y\in OX_2.$ Since $Y$ is the one with the smallest distance to the axis $OX_3,$ then there aren't any other vertices on the plane $OX_2X_3.$ If the segment $\overline{W_1Y}$  is an edge of some unexceptional face, so condition ii) is fulfilled for this face. Otherwise the segment $\overline{W_1Y}$ is edge of some exceptional face with respect to $OX_1.$ Hence $W_1$ is at distance $1$ to the axis $OX_1.$ So $W_1=(1,0,1)$ and $\overline{W_1Y}=I^k_2,$ where $k=x_2(Y),$ which by Proposition \ref{k2} isn't possible (Fig. \ref{lemwyb}, $2^0$b).
\\\bf c) \rm $Y\in OX_3.$ Then $Y=W_3.$ If the segment $\overline{W_1W_3}$ is an edge of some unexceptional face, so condition ii) is fulfilled for this face. Otherwise the segment $\overline{W_1W_3}$ is an edge of some exceptional face with respect to the axis $OX_1$ or $OX_3,$ because it can't be an edge of any exceptional face with respect to the axis $OX_2.$
\par If it is an edge of exceptional face with respect to $OX_1,$ then there exists a vertex $B\in OX_1X_2$ at distance $1$ to the  axis $OX_1.$ By Lemma \ref{k4} the segment $\overline{BW_3}$ is an edge of some unexceptional face, so condition i) is fulfilled for this face.
\par If it is an edge of exceptional face with respect to $OX_3,$ then there exists a vertex $W_2\in OX_2X_3$ at distance $1$ to the axis $OX_3.$ Hence by Lemma \ref{k4} there exists $i\in\{1,2\}$ and a vertex $A_i\in OX_iX_3$ such that the segment $\overline{A_iW_{3-i}}$ is an edge of some unexceptional face. Therefore condition ii) is fulfilled for this face. That concludes the proof.
\end{proof}


\section{The proof of the main result.}

We go now to the proof of the main result. For the convenience of the reader we repeat it once again.

\vspace{0.2cm}

\par\noindent\bf Theorem \ref{n2} \rm \it Let $f:\left( \mathbb {C}^3,0\right)\longrightarrow \left( \mathbb {C},0\right)$ be an isolated and nondegenerate singularity.  
\vspace{0.15cm}

\par\noindent $1^0$ If $\Gamma^{2}(f)=\emptyset$ or $\Gamma^{2}(f)=E_f,$ then there exists exactly one segment  $I\in\mathcal{J}\cap\Gamma^{1}(f)$ and $${\pounds}_0(f)=m(I)-1.$$
\par \noindent $2^0$ If $\Gamma^2(f)\setminus E_f\neq\emptyset,$ then \begin{equation} {\pounds}_0(f)\leq\max _{S\in\Gamma^{2}(f)
 \setminus E_f} m(S)-1. \end{equation} 
\vspace{0.2cm} \rm
 
\begin{proof}
\\$\mathbf 1^0.$ 
If $\Gamma^{2}(f)=\emptyset$ or $\Gamma^{2}(f)=E_f\neq \emptyset,$ then by Theorem \ref{n19} and Proposition \ref{k2} there exists exactly one segment $I\in\Gamma^1(f)\cap \mathcal{J}.$ Without loss of generality we can assume that $I=I_3^k$ for some $k\in\{2,3,\ldots\}.$ Then $(1,1,0)\in\supp (f)$ and hence we can find a monomial of the form $az_1z_2,\,a\neq 0$ in the expansion of $f.$ Therefore $f''_{z_1z_2}(0)= a \neq 0.$ Observe that $(2,0,0)\not\in \supp (f)$ or $(0,2,0)\not\in \supp (f).$ Otherwise the point $(1,1,0)$ would be in the interior of the segment $\overline{(2,0,0)(0,2,0)},$ which would contradict that $(1,1,0)\in\Gamma^0(f).$ So $f''_{z_1z_1}(0)=0$ or $f''_{z_2z_2}(0)=0.$ Summing up
$$\det\left[\frac{\partial^2 f}{\partial z_i\partial z_j}\right]_{i,j=1}^2(0)\neq 0,\quad \mbox{thus}\quad \rank\left[\frac{\partial^2 f}{\partial z_i\partial z_j}\right]_{i,j=1}^3(0)\geq 2.$$ 
Then by Lemma \ref{milnor} we get that $\pounds_0(f)=\mu_0(f).$ We have the following cases.
\\\bf a) \rm The Newton diagram $\Gamma_+(f)$ is convenient. Then $\Gamma^{2}(f)\neq\emptyset,$ so by the assumption we have that $\Gamma^2(f)=E_f.$ Since $f$ is nondegenerate, then by Kouchnirenko theorem (\cite{K}, Thm. I) the Milnor number $\mu_0(f)$ is equal to the Newton number $\nu (f).$ By definition $$\nu (f)=3! V_3-2! V_2+V_1-1,$$ where $V_3$ is the volume of the set $\mathbb{R}_{+}^3\setminus \inter(\Gamma_+(f)),$ and $V_k,\,k=1,2$ are $k$-dimensional Lebesgue measures of the intersection of this set and sum of linear subspaces of dimension $k$ spanned by the coordinate axes. 

\begin{figure}[!ht]
\begin{center}
 \includegraphics{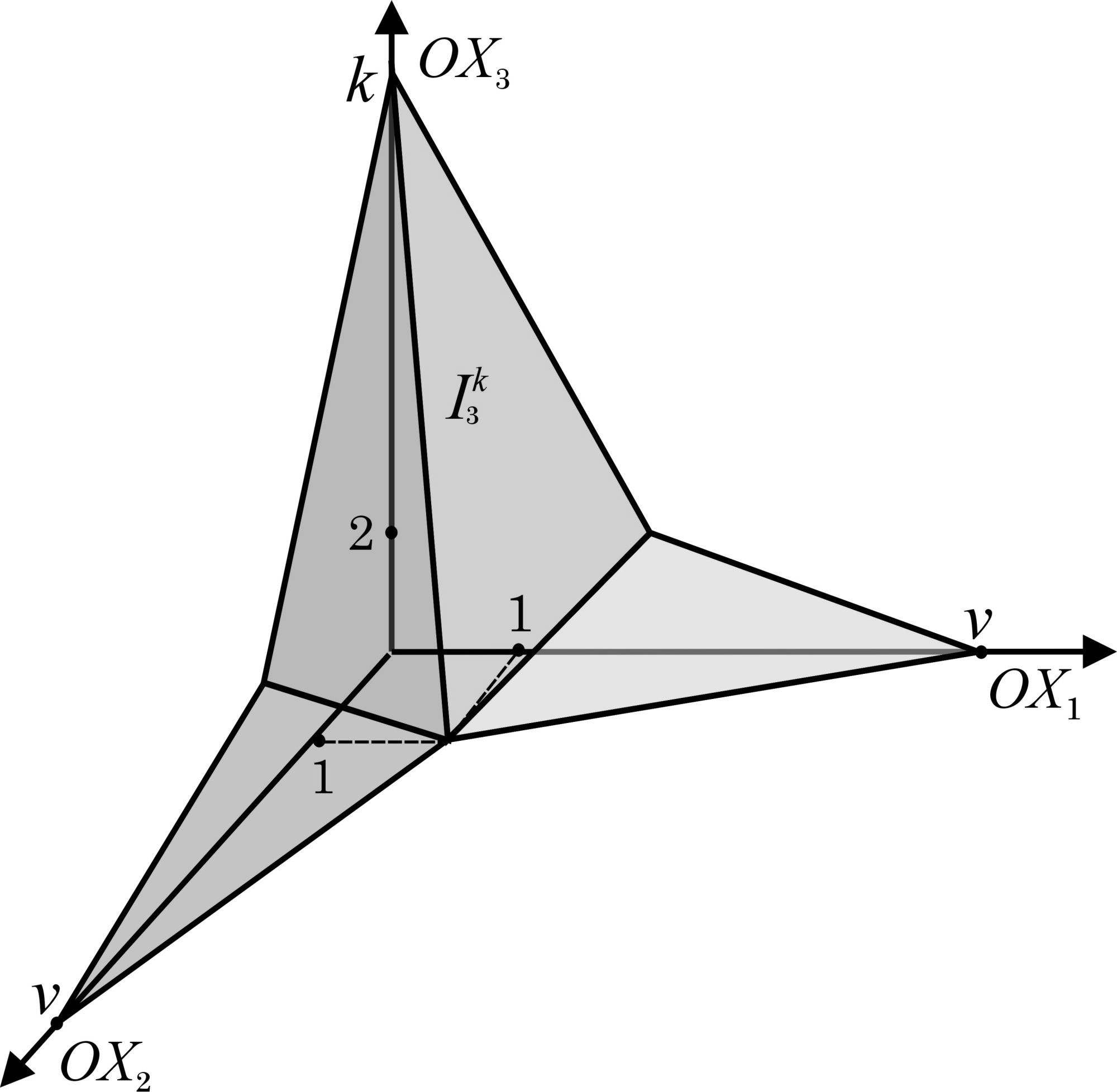} 
  \caption{$\Gamma^2(f)=E_f$ and $I^k_3\in\Gamma^1(f).$}\label{kusznirenko}
  \end{center}
\end{figure}



 It is not difficult to check that $\nu(f)=k-1=m(I)-1$ (see Fig. \ref{kusznirenko}).
Summing up we get $$\pounds_0(f)=\mu_0(f)=\nu(f)=m(I)-1,$$ which finishes the proof in this case.
\\\bf b) \rm If $\Gamma_+(f)$ isn't convenient, then we deform $f$ to get an isolated singularity, which has convenient the Newton diagram. To that end we define new singularity $$g(z_1,z_2,z_3):=f(z_1,z_2,z_3)+\alpha_1z_1^v+\alpha_2z_2^v,$$ with $\alpha_i=0,$ if $\Gamma_+(f)\cap OX_i\neq \emptyset$ or else $\alpha_i=1,$  $i=1,2.$ We choose a number $v\in\{2,3\ldots\}$ to fulfill the following conditions : 
\\i) $\ord(\nabla g-\nabla f)>\pounds_0(f),$  
\\ii) if $E_f\neq \emptyset,$ then $v>\max_{S\in E_f}m(S),$
\\iii) if $\Gamma^2(f)= \emptyset,$ then $v>m(I).$ 
\par\noindent Because $(0,0,m(I))\in\supp(f),$ so by definition of $g$ we get that $\Gamma_+(g)$ is convenient. Since $\ord(\nabla g-\nabla f)>\pounds_0(f),$ then by Lemma 1.4 in \cite{P} we have that $g$ is an isolated singularity and 
\begin{equation} \label{n22} \pounds_0(g)=\pounds_0(f). \end{equation}
\par Moreover $\Gamma(g)=\Gamma(f)\cup T,$ where $T$ is the set of such faces $S\in\Gamma(g)\setminus\Gamma(f)$ that $(v,0,0)\in S,$ (if $\alpha_1=1$) or $(0,v,0)\in S,$ (if $\alpha_2=1$). Observe that two-dimensional faces of the set $T$ are exceptional faces of $g.$ Hence and since $\Gamma^2(f)=\emptyset$ or $\Gamma^2(f)=E_f$ we get that $\Gamma^2(g)=E_g.$ It is easy to check that for every face $S\in T$ there exists $i\in\{1,2,3\}$ such that $(f_S)_{z_i}$ is a monomial. Therefore $g$ is nondegenarate on every face $S\in T.$ Then by nondegeneracy of $f$ and by the equality $\Gamma(g)=\Gamma(f)\cup T$ we get that $g$ is nondegenerate. Hence by proof of the case a) used for $g$ we have that $\pounds_0(g)=m(I)-1.$ Summing up by (\ref{n22}) we get that $$\pounds_0(f)=\pounds_0(g)=m(I)-1.$$ It finishes the proof in this case.
\vspace{0.2cm}

\noindent $\mathbf 2^0.$ If $\Gamma^2(f)\setminus E_f\neq\emptyset,$ then by Lemma about the choice of an unexceptional face we choose the face $S_i\in\Gamma^2(f)\setminus E_f$ for every axis $OX_i,\,i=1,2,3$ such that is fulfilled at least one of the two conditions:
\vspace{0.1cm}
\\ i)  there exists a point $W \in OX_i,$ which is a vertex of the face $S_i,$
\\ ii) there are exist $j,k\in\{1,2,3\}\setminus \{i\},\,j \neq k$ and vertices: $W \in OX_iX_j$ at distance $1$ to the axis $OX_i$ and $Y \in OX_iX_k$ such that the segment $\overline{WY}$ is an edge of the face $S_i.$ 
\par We show that $\pounds_0(f)\leq \max_{i=1}^3 m(S_i)-1.$ Suppose to the contrary that $$\pounds_0(f)>\max_{i=1}^3 m(S_i)-1.$$ By Property b) of the \L ojasiewicz exponent (p. \pageref{b}) there exists a parameterization $\phi=(\phi_i)_{i=1}^3\in\mathbf{C}\{t\}^3$ such that
\begin{equation}\label{n52} \frac{\ord(\nabla f\circ\phi)}{\ord\phi}>\max_{i=1}^3m(S_i)-1. \end{equation} 
We have the following cases.
\\\bf a) \rm $\phi_i\neq 0,\,i=1,2,3.$ Denote by $L$ the supporting hyperplane to $\Gamma_+(f)$ such that $L\bot(\ord\phi_i)_{i=1}^3.$ 
Then by Lemma \ref{n6} we get that 
$$\frac{\ord(\nabla f\circ\phi)}{\ord\phi}\leq m(L)-1.$$
This and inequality (\ref{n52}) shows that $m(L)>\max_{i=1}^3m(S_i).$ Without loss of generality we can assume that $m(L)=x_1(L).$ Hence we get that
\begin{equation}\label{p3} x_1(L)>m(S_1)\geq x_1(T), \end{equation} 
where $T$ is the supporting plane to the face $S_1.$ By the inequality (\ref{p3}) the condition i) for face $S_1$ isn't possible. Thus the condition ii) is fulfilled  for this face and without loss of generality we may assume that $j=3$ in this condition. Then there are vertices: $W \in OX_1X_3$ at distance $1$ to the axis $OX_1$ and $Y \in OX_1X_2$ such that the segment $\overline{WY}$ is the edge of the face $S_1$ (Fig. \ref{n60}). We show that there exists a plane $K\parallel L,$ which support $\Gamma_+(f'_{z_3})$ exactly in one point $W-1_3\in OX_1$ and $m(K)\leq m(S_1)-1.$

\begin{figure}[!ht]
\begin{center}
 \includegraphics{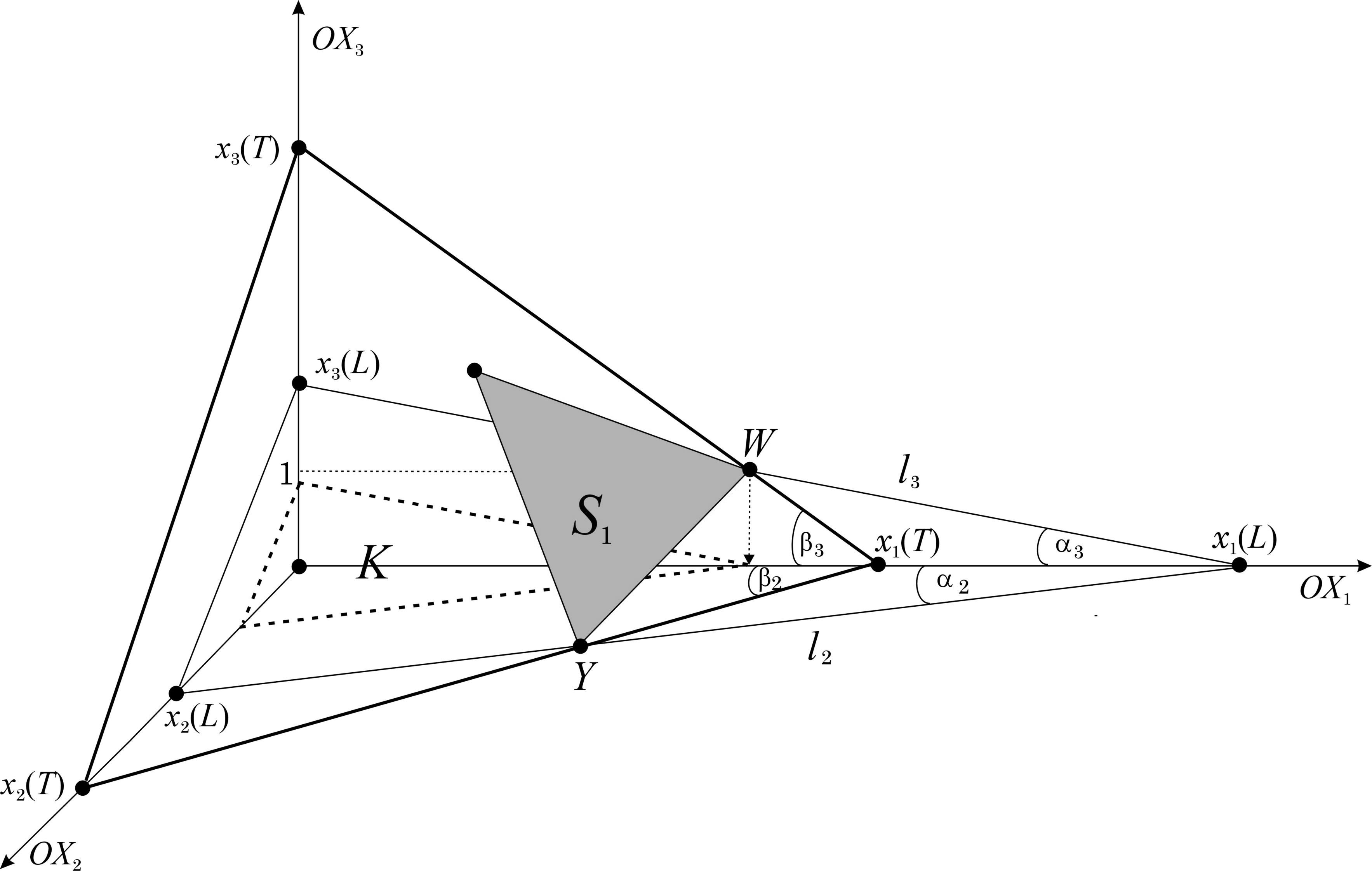} 
  \caption{An unexceptional face $S_1.$} \label{n60}
  \end{center}
\end{figure}

\par For $i=2,3$ we will denote by $l_i$ the line $L\cap OX_1X_i$ and by $\alpha_i$ the acute angle between the line $l_i$ and the axis $OX_1,$ and by $\beta_i$ the acute angle between the line $T\cap OX_1X_i$ and the axis $OX_1.$ Since $L$ is a supporting plane to $\Gamma_+(f),$ then $W$ lies on the line $l_3$ or above it and $Y$ lie on the line $l_2$ or above it. Hence and by (\ref{p3}) we get that $\alpha_i<\beta_i,\,i=2,3.$ Let now $K\parallel L,$ be the plane such that $W-1_3\in K.$ Since the set $\supp f$ lie on the plane $T$ or above it, then $\supp f'_{z_3}$ lie in the plane $T-1_3$ or above it. Because $\alpha_i<\beta_i,\,i=2,3$ and $K\parallel L,$ so $x_i(K)< x_i(T-1_3),\,i=2,3$ and the more $\supp f'_{z_3},$ besides the point $(W-1_3)\in K,$ lie above the plane $K.$ Therefore the plane $K$ supports $\Gamma_+(f'_{z_3})$ exactly in one point $W-1_3\in OX_1.$ Moreover $m(K)\leq m(T-1_3).$ This and Property \ref{n4'} shows that 
\begin{equation}\label{n59}m(K)\leq m(T-1_3)\leq m(T)-1=m(S_1)-1. \end{equation} 
Summing up by inequality (\ref{n59}) and by Property \ref{n8}b we have

$$\frac{\ord(\nabla f\circ\phi)}{\ord\phi}\leq \frac{\ord( f'_{z_3}\circ\phi)}{\ord \phi}=m(K)\leq m(S_1)-1,$$

\noindent which leads to a contradiction with inequality (\ref{n52}).
\vspace{0.1cm}

\par\noindent\bf b) \rm There exists $i\in\{1,2,3\}$ such that $\phi_i= 0.$ Without loss of generality we can assume that $i=1.$ Hence $\phi=(0,\phi_2,\phi_3).$ Denote $\phi_0=(\phi_2,\phi_3).$ We  represent the singularity $f$ in the form$$f(z_1,z_2,z_3)=g(z_2,z_3)+z_1h(z_2,z_3)+z_1^2h_2(z_1,z_2,z_3),\,\,\mbox{where}\,\,g,h\in\mathcal{O}^2,\,h_2\in\mathcal{O}^3.$$
Since $f$ is an isolated singularity, then $g\neq 0$ (see Property \ref{n17}) and thus $\Gamma(g)\neq \emptyset.$ Moreover $g(0)=h(0)=0,\,\nabla g(0)=0,$ or $g$ is a singularity (not necessarily isolated). Because $f$ is nondegenerate and $\Gamma(g)=\{S\in \Gamma(f):S\subset \{x_1=0\}\}$ (see Property \ref{w4}), so $g$ is nondegenerate. Summing up $g$ is nondegenarate singularity. 
\par Suppose first that $\phi_2=0$ and $\phi_3 \neq 0$ (the case $\phi_3=0$ and $\phi_2 \neq 0$ consider analogously). It is easy to observe that in each case i) or ii) by the choice of the face $S_3$ and the vertex $W,$ there exists $i\in\{1,2,3\}$ such that $W-1_i\in OX_3.$ Then from Property \ref{n4'} we have 
\begin{eqnarray} \nonumber  \frac{\ord(\nabla f\circ\phi)}{\ord\phi} & \leq & \frac{\ord(f'_{z_i}\circ\phi)}{\ord\phi}=\frac{x_3(W-1_i)\ord\phi_3}{\ord\phi_3}=x_3(W-1_i) \leq \\ & \leq &  m(T-1_i)\leq m(T)-1, 
\end{eqnarray} 
where $T$ is the supporting plane to the face $S_3.$ The last inequality contradicts the inequality (\ref{n52}). 
\par Now suppose  that $\phi_2\neq 0$ and $\phi_3\neq 0.$ Denote $w=(\ord\phi_2,\ord\phi_3).$ Consider the unique supporting line $l\subset OX_2X_3$ to $\Gamma_+(g)\subset OX_2X_3$ such that $w\bot l.$ Then by Lemma \ref{n6} we have 
$$\frac{\ord(\nabla f\circ\phi)}{\ord\phi}\leq \frac{\ord(f'_{z_2}\circ\phi,f'_{z_3}\circ\phi)}{\ord\phi}=\frac{\ord(\nabla g\circ\phi_0)}{\ord\phi_0} \leq m(l)-1.$$  
This and inequality (\ref{n52}) shows that $m(l)>\max_{i=1}^3m(S_i).$ Without loss of generality we may assume that $m(l)=x_3(l).$ Hence we obtain that
\begin{equation}\label{n53} x_3(l)>m(S_3)=m(T)\geq x_3(T). \end{equation} 
Then for the face $S_3$ the condition i) can't be true. So the condition ii) holds for it. Choose $j,k\in\{1,2\},\,j\neq k$ and vertices: $W_j \in OX_3X_j$ at distance $1$ from the axis $OX_1$ and $Y \in OX_3X_k$ such that the segment $\overline{W_jY}$ is an edge of the face $S_3.$ We shall show that there is a line $k_j\subset OX_2X_3,\,k_j\parallel l,$ which supports $\Gamma_+(f'_{z_j}(0,z_2,z_3))\subset OX_2X_3$ excatly in one point $W_j-1_j$ and $m(k_j)\leq m(S_3)-1.$

\begin{figure}[!ht]
\begin{center}
 \includegraphics{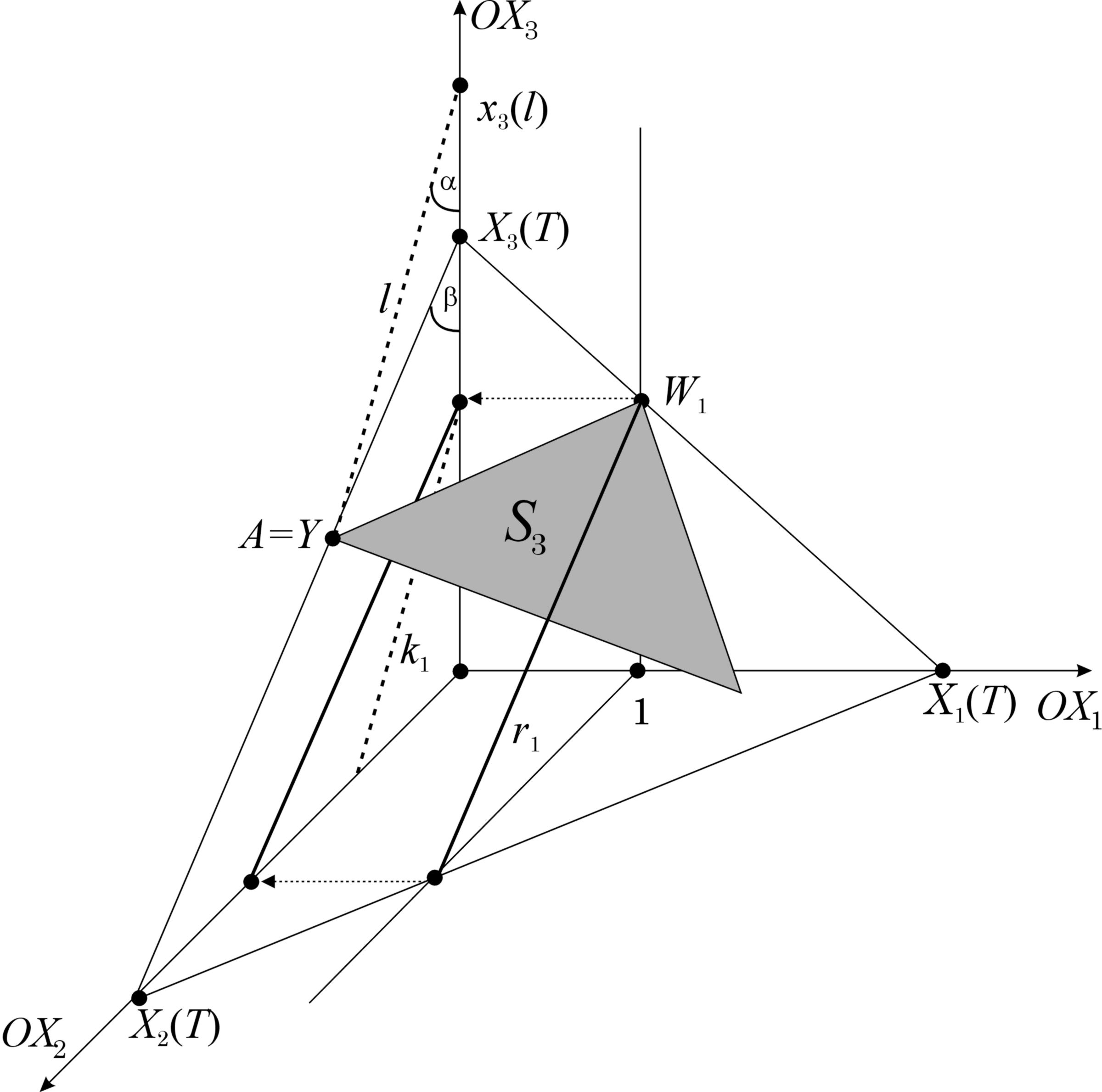} 
  \caption{$j=1$}\label{n54}
  \end{center}
\end{figure}

\begin{figure}[!ht]
\begin{center}
 \includegraphics{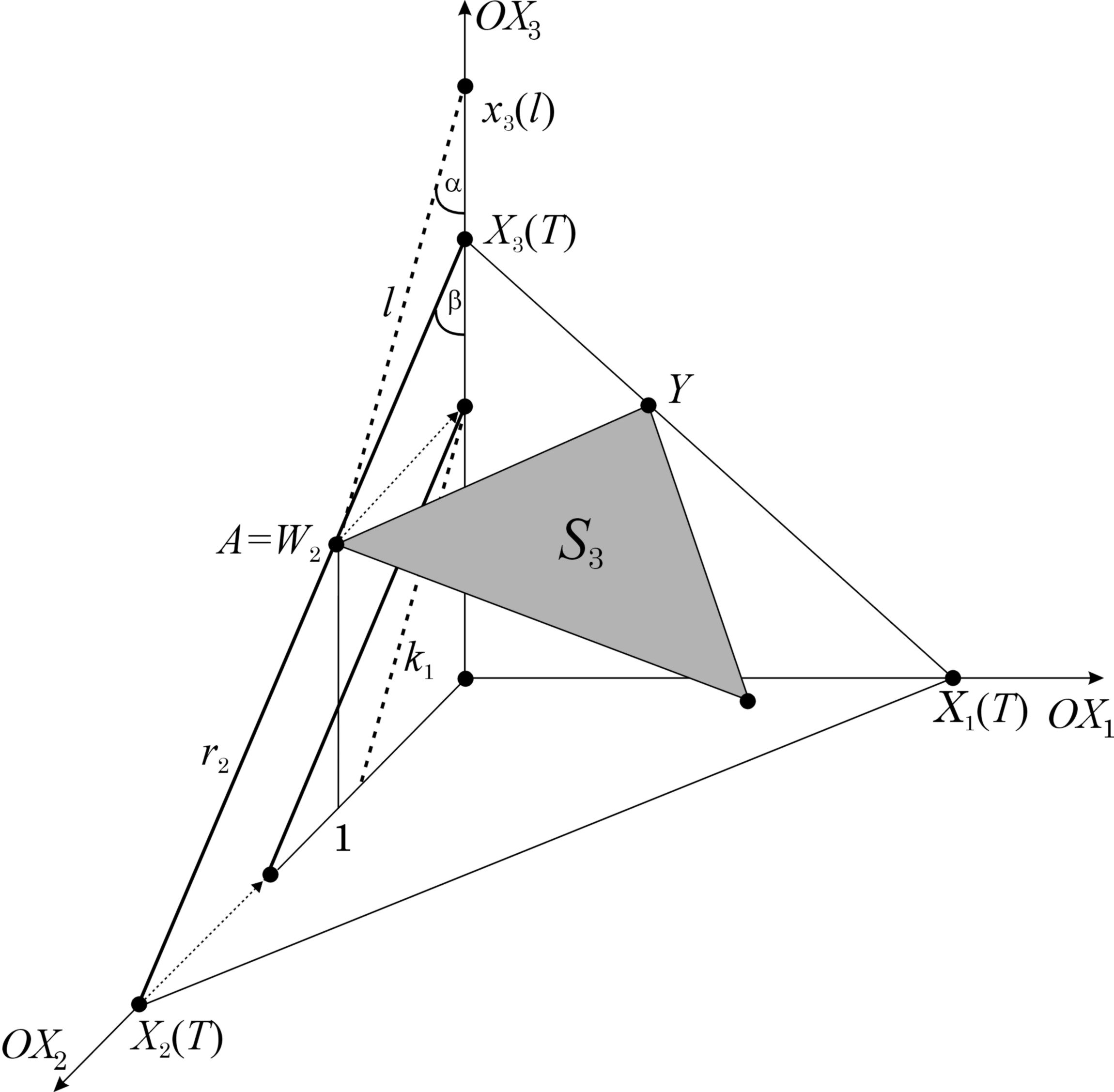} 
  \caption{$j=2$}\label{n57}
  \end{center}
\end{figure}

\par Denote $r_j=T\cap\{x_1=2-j\}$ and let $A$ be the vertex of the edge $\overline{W_jY},$ which lies on the plane $OX_2X_3.$ Let $\alpha$ be the acute angle between the line $l$ and the axis $OX_3$ and $\beta$ the acute angle between the line $T\cap OX_2X_3$ and the axis $OX_3.$ Since $l$ is a supporting line to $\Gamma_+(g),$ then $A$ lies on the line $l$ or above it (in Fig. \ref{n54} and Fig. \ref{n57} $A\in l$). This and (\ref{n53}) shows that $\alpha<\beta.$ Consider now the line $k_j\parallel l$ such that $W-1_j\in k_j.$ Because the set $\supp f$ lies on the plane $T$ or above it, so the set $\supp f'_{z_j}(0,z_2,z_3)$ lies on the line $r_j-1_j$ or above it. Since $\alpha<\beta$ and $k_j\parallel l,$ then $x_2(k_j)<x_2(r_j-1_j)$ and the more $\supp f'_{z_j}(0,z_2,z_3)$ besides the point $(W-1_j)\in k_j,$ lies above the line $k_j.$ Therefore the line $k_j$ supports $\Gamma_+(f'_{z_j}(0,z_2,z_3))$ exactly in one point $W_j-1_j.$ Moreover $m(k_j)\leq m(r_j-1_j).$ Hence and by Property \ref{n4'} we get that
\begin{equation}\label{n58}m(k_j)\leq m(r_j-1_j)\leq m(T-1_j)\leq m(T)-1=m(S_3)-1. \end{equation} 
Summing up by inequality (\ref{n58}) and by Property \ref{n8}b we have that

$$\frac{\ord(\nabla f\circ\phi)}{\ord\phi}\leq \frac{\ord( f'_{z_j}\circ\phi)}{\ord \phi}\leq \frac{\ord( f'_{z_j}(0,\phi_0))}{\ord\phi_0}=m(k_j)\leq m(S_3)-1,$$

\noindent which leads to a contradiction with inequality (\ref{n52}). 
\par Therefore we obtain $$\pounds_0(f)\leq\max_{i=1}^3 m(S_i)-1\leq \max_{S\in\Gamma^2(f)\setminus E_f}m(S)-1,$$ which completes the proof of part $2^0.$ 
\end{proof}

\section{Examples and open problems.}

We now give examples to illustrate the main result of this paper (Thm. \ref{n2}). The first example illustrates part $1^0$ of Theorem \ref{n2}. For the singularity of this example we have $\Gamma^2(f)=E_f$ and $\pounds_0(f)<\max_{S\in E_f}m(S)-1.$  

\begin{example}\label{n40} Let $f(z_1,z_2,z_3):=z_1^3+z_2^3+z_3^2+z_1z_2.$ In this case $\Gamma^2(f)=E_f,$ thus, under part $1^0$ of Theorem \ref{n2} we have that $I_3^2\in\Gamma^1(f)\cap\mathcal{J}$ and $\pounds_0(f)=m(I_3^2)-1=1<2=\max_{S\in E_f}m(S)-1.$   
\end{example} 

The next example also illustrates part $1^0$ of Theorem \ref{n2}. For the singularity in this example we have $\Gamma^2(f)=\emptyset$ and $\Gamma^1(f)=\{I_2^3\}.$

\begin{example}\label{n41} Let $f(z_1,z_2,z_3):=z_2^3+z_1z_3+z_1^2z_3^2.$ In this case $\Gamma^2(f)=\emptyset.$ Thus, under part $1^0$ of Theorem \ref{n2} we have that $I_2^3\in\Gamma^1(f)\cap\mathcal{J}$ and $\pounds_0(f)=m(I_2^3)-1=2.$ 
\end{example} 

The last example illustrates part $2^0$ of Theorem \ref{n2}. It shows that for the singularity in this example, the estimate obtained from this part of the theorem is optimal, i.e. in formula (\ref{n2'}) we have the equality.
\begin{example}\label{n42} Let $f(z_1,z_2,z_3):=z_1^3+z_2^3+z_1z_3^4+z_2z_3^4.$ In this case $\Gamma^2(f)=\{S_1,S_2\},\,S_1=\conv \{(3,0,0),(0,3,0),(1,0,4),(0,1,4)\},\,S_2=\conv \{(1,0,4),(0,1,$ \\ $,4),(0,0,20)\},\,E_f=\{S_2 \}.$ Hence from part $1^0$ of Theorem \ref{n2} we get that 
$\pounds_0(f)\leq m(S_1)-1=5$ (by Fukui theorem \cite{F} we would get $\pounds_0(f)\leq\max(m(S_1),$\\$m(S_2))-1=19$).  In this example one can show that $\pounds_0(f)=5.$   
\end{example}

The above example suggests the following is true.
\begin{conjecture} \label{n44} Let $f:\left( \mathbb {C}^n,0\right)\longrightarrow \left( \mathbb {C},0\right),\,n\geq 2, $ be an isolated and nondegenerate singularity such that $\Gamma^{n-1}(f)\setminus E_f\neq\emptyset.$ Then 
\begin{equation} {\pounds}_0(f)=\max _{S\in\Gamma^{2}(f)\setminus E_f} m(S)-1. \end{equation}
\end{conjecture} \rm                                            

Naturally arises also the question how to generalize the part $1^0$ of Theorem \ref{n2} to $n$-dimensional case for $n>3 $.

\begin{problem}\label{n45} Characterize isolated singularities in $n$-variables, $n>3,$ for which $\Gamma^{n-1}(f)=\emptyset$ or $\Gamma^{n-1}(f)=E_f$ and give the formula for the \L ojasiewicz exponent of such singularities.
\end{problem}


\end{document}